\documentclass[12pt,reqno]{amsart}

\usepackage{amssymb}

\newtheorem{theorem}{Theorem}[section]
\newtheorem{proposition}[theorem]{Proposition}

\newtheorem{corollary}[theorem]{Corollary}

\numberwithin{equation}{section}

\begin{document}

\title[Equality of algebraic and geometric ranks of Cartan subalgebras]{Equality of the
algebraic and geometric ranks of Cartan subalgebras and
applications to linearization of a system of ordinary differential equations}

\author[H. Azad]{Hassan Azad}

\address{Department of Mathematics and Statistics, King Fahd University,
Saudi Arabia}

\email{hassanaz@kfupm.edu.sa}

\author[I. Biswas]{Indranil Biswas}

\address{School of Mathematics, Tata Institute of Fundamental Research, Homi
Bhabha Road, Bombay 400005, India}

\email{indranil@math.tifr.res.in}

\author[F. M. Mahomed]{Fazal M. Mahomed}

\address{DST-NRF Centre of Excellence in Mathematical and Statistical Sciences,
School of Computer Science and Applied Mathematics, University of the Witwatersrand,
Johannesburg, Wits 2050, South Africa}

\email{Fazal.Mahomed@wits.ac.za}



\date{}

\begin{abstract}
If $L$ is a semisimple Lie algebra of vector fields on ${\mathbb R}^N$ with a split 
Cartan subalgebra $C$, then it is proved that the dimension of the generic orbit of 
$C$ coincides with the dimension of $C$. As a consequence one obtains a local 
canonical form of $L$ in terms of exponentials of coordinate functions and vector 
fields that are independent of these coordinates -- for a suitable choice of 
coordinates. This result is used to classify semisimple algebras of vector fields on 
${\mathbb R}^3$ and to determine all representations of ${\rm sl}(N, {\mathbb R})$ 
as vector fields on ${\mathbb R}^N$. These representations are used to find 
linearizing coordinates for any second order ordinary differential equation that 
admits ${\rm sl}(3, {\mathbb R})$ as its symmetry algebra and for a system of two 
second order ordinary differential equations that admits ${\rm sl}(4, {\mathbb R})$ 
as its symmetry algebra.
\end{abstract}

\maketitle

\tableofcontents

\section{Introduction}

Our aim here is to illustrate a general observation -- namely, that the presence
of large abelian subalgebras of the symmetry algebra of differential equations -- or
systems thereof -- is behind the various schemes for finding coordinates that linearize
a given linearizable equation.

This is illustrated in detail for a single ordinary differential equation and a system
of two ordinary differential equations.
This is achieved by investigating the local canonical form of semisimple algebras of
vector fields on ${\mathbb R}^N$.

As is well known, the symmetry algebra of the equation $y''(x) \,=\, 0$ is
${\rm sl}(3, {\mathbb R})$ and the
symmetry algebra of the system of two equations $y''(x) \,=\, 0$, $z''(x) \,=\, 0$ is
${\rm sl}(4, {\mathbb R})$; see \cite{Li}, \cite{Ib}.
More generally, the symmetry algebra of a system of $n$ free particles is
${\rm sl}(n+2, {\mathbb R})$ \cite{GL}, \cite{AKM}.

Using the local canonical forms of semisimple algebras of vector field, the
representations of ${\rm sl}(n, {\mathbb R})$ as vector fields on ${\mathbb R}^{n-1}$
are determined. Declaring one variable as the independent and the remaining as
dependent variables and computing the invariants in the second prolongations gives
the canonical forms of the equations considered as well as the linearizing
coordinates coming from abelian subalgebras of maximal dimension and maximal rank.

The methods used give at the same time a list, in terms of root systems, of
semisimple subalgebras of vector fields on ${\mathbb R}^{3}$. The classification of all
subalgebras of vector fields on ${\mathbb R}^{3}$ is one of the main topics in
Ch. 1--4 of \cite{Li2}; the reader is referred to the survey \cite{Ko}.

We have organized this paper in two parts. Section \ref{se2} deal with general
properties of semisimple Lie algebras of vector fields, and the rest deals with
applications to linearization.

\section{Geometric rank of Cartan subalgebras}\label{se2}

The geometric rank of a Lie algebra of vector fields is the maximum of the
dimensions of the leaves of the corresponding foliation. It is same as the
dimension of the generic leaf. The main result of this section is as follows:

\begin{theorem}\label{thm1}
Let $L$ be a semisimple Lie algebra of vector fields on ${\mathbb R}^N$ which has a
split Cartan subalgebra $C$. Then the dimension of $C$ equals the geometric rank of $C$.

Moreover, in suitable coordinates $x_1, \, x_2,\, \cdots ,\, x_N$, the root spaces
corresponding to a given system of simple roots and their negatives are of the form
$\exp(x_i)V_i$, $\exp(-x_i)W_i$, $i\,=\,1,\, \cdots,\, n$, where $V_i$ and $W_i$ are
vector fields whose coefficients with respect to the basis $\partial_{x_1},\,
\cdots,\, \partial_{x_N}$ are independent of $x_1,\, \cdots,\, x_n$,
and the linear span of the system of vector fields
$$
[\exp(x_i)V_i,\, \exp(-x_i)W_i]\,=\, H_i\, , \ \ \ i\,=\, 1,\, \cdots,\, n
$$
is that of $\{\partial_{x_1},\, \cdots,\, \partial_{x_n}\}$.

For any $1\,\leq\, i\,\leq\, n$, the Lie algebra generated by $\exp(x_i)V_i,\,
\exp(-x_i)W_i$ is a copy of ${\rm sl}(2,{\mathbb R})$.
\end{theorem}

Theorem \ref{thm1} is very useful in finding -- up to local equivalence -- semisimple
algebras of vector fields in ${\mathbb R}^N$ for $N \,\leq\, 3$.
The local classification of all Lie algebras of vector fields in
${\mathbb R}^N$, $N \,\leq\, 3$,
is one of the main themes in S. Lie's work \cite{Li2}.

The importance of such a classification is in the canonical forms of differential
equations for which Lie invented his theory of prolongations and differential
invariants.

For an abelian algebra, the geometric rank and dimension, in general, are not equal.
For example the linearly independent vector fields $\partial_x,\, y\partial_x,\,
y^2\partial_x,\,\cdots,\, y^k\partial_x$ commute for every $k$.
Also, for any Lie algebra of vector fields on ${\mathbb R}^N$ the dimension of its
generic orbit is the same as the geometric rank of a certain abelian algebra of
vector fields -- as shown in \cite{ABGM}.

Recall that if $L$ is a Lie algebra of vector fields on ${\mathbb R}^N$ then the
generic rank of the matrix of coefficients of a basis of $L$ in the standard basis
$\partial_{x_1},\,\cdots,\, \partial_{x_N}$
of vector fields on ${\mathbb R}^N$ is the dimension of a generic leaf for $L$.
It is thus an invariant of the algebra $L$.

\subsection{Preliminaries and notation}\label{sec2.2}

The proof of Theorem \ref{thm1} uses the following standard results about roots of a Lie
algebra, for which the reader is referred to \cite{Bo}, \cite{HN}, \cite{Kn}

Let $C$ be a split Cartan subalgebra of $L$ and $R$ its root system. Let
$S \,=\,\{\alpha_1,\, \cdots,\, \alpha_n\}$ be the simple roots of $C$ for a choice of
positive roots. The root spaces are
subspaces of $L$ normalized but not centralized by $C$. The corresponding linear
functions are the roots of $C$. Each root space is one dimensional and a nonzero vector
in a root space is called a root vector. For each root $r$ we choose a nonzero element
$X_r$ in the corresponding root space. The algebra $L$ is generated as a vector space by
the Cartan subalgebra $C$ and the root vectors $\{X_r\}_{r\in R}$ with
$$[X_r,\, X_s]\,=\, N_{r,s} X_{r+s}\, ,$$
where $N_{r,s}\, \not=\, 0$ if and only if $r + s$ is a root. Moreover
$[X_r,\, X_{-r}]\,=\, H_r$ is a nonzero element of $C$ and the Lie algebra
generated by the pair of root vectors $\{X_r,\,X_{-r}\}$ is a copy of
${\rm sl}(2,{\mathbb R})$.

The root vectors $X_{\alpha_i},\, X_{-\alpha_i}$, $i \,=\, 1,\,\cdots,\, n$, generate
$L$ as a Lie algebra. For a simple root $\alpha_i$, we denote the element
$[X_{\alpha_i},\, X_{-\alpha_i}]$ by $H_i$.

\subsection{Proof of Theorem \ref{thm1}}

Let $n$ be the dimension of the Cartan subalgebra $C$. If $n \,=\, 1$ then
-- using the notation of Section \ref{sec2.2} -- in a neighborhood of a point where
$H_1$ is not zero, we can find coordinates $x_1,\, \cdots,\, x_N$
in which $H_1 = \, \partial_{x_1}$. The vector
fields $V$ which are eigenvectors of $H_1$ in the sense that
$[H_1,\, V] \,=\,\lambda\cdot V$ are of the
form $V\, =\, \exp(\lambda x_1) U$, where $U$ is a vector field whose coefficients in the
basis $\partial_{x_1},\, \cdots,\, \partial_{x_N}$
are independent of $x_1$. Our algebra is generated by eigenvectors of
$H_1$ for nonzero and opposite eigenvalues. Thus if we substitute $\lambda x_1$
in place of $x_1$, and leave the
other coordinates unchanged, then in this coordinate system the algebra is generated by
vector fields $\exp(x_1) V_1,\, \exp(-x_1) W_1$, where $V_1$ and $W_1$ are vector fields
whose coefficients in the basis $\partial_{x_1},\,
\cdots,\, \partial_{x_N}$ are independent of $x_1$.

We will employ induction. Assume that the theorem is proved for all $C$
with $\dim C \,\leq\, m$.

Now let the dimension of the Cartan algebra $C$ be $n \,= \,m + 1$.

The algebra generated by the root vectors $\{X_{\alpha_i},\,
X_{-\alpha_i}\}_{i=1}^m$ is semisimple. By the induction
hypothesis, the rank and dimension of the system of vector fields
$\{H_i\}_{i=1}^m$ is $m$. As these are commuting vector fields, we
can introduce coordinates in which $H_i\,=\, \partial_{x_i}$, $1\,
\leq\, i\,\leq\, m$. Moreover, again by the induction hypothesis,
root vectors of this subalgebra corresponding to the simple roots
and their negatives are of the form $\exp(x_i) V_i,\, \exp(-x_i)
W_i$, $1\, \leq\, i\,\leq\, m$, where $V_i$ and $W_i$ are vector
fields whose coefficients in the basis $\partial_{x_1},\,
\cdots,\, \partial_{x_N}$ are independent of $x_1,\, \cdots,\, x_m$.

If the rank of $H_1,\, \cdots,\, H_m,\, H_{m+1}$ is less than $m + 1$, then as
$H_{m+1}$ commutes with $H_1,\, \cdots,\, H_m$, it can be written as
$$
H_{m+1}\, =\, f_1\partial_{x_1}+\ldots + f_m\partial_{x_m}
$$
with $\partial_{x_j}f_i\,=\, 0$ for all $1\, \leq\, i,\, j\, \leq\, m$.
Not all the coefficient functions $f_i$, $i \,=\, 1,\,\cdots, \, m$,
can be constant because $H_1,\, \cdots,\, H_{m+1}$
are linearly independent, so say $f_\ell$ is not a constant function.

The root vector $X_\ell$ can be written as
$$
X_\ell \,=\,\exp(x_\ell)V_\ell\, ,
$$
where $V_\ell$ is a vector field whose coefficients in the basis
$\partial_{x_1},\,
\cdots,\, \partial_{x_N}$ are independent of $x_1,\, \cdots,\, x_m$.

We write $$V_\ell\,=\, U_1+U_2\,$$ where $U_1$ and $U_2$ are vector fields with
$$
U_1\,=\,\sum_{i=1}^m g_i\partial_{x_i}\ \ \ \text{ and }\ \ \
U_2\,=\,\sum_{i=m+1}^N g_i\partial_{x_i}
$$
such that all $g_i$ are independent of $x_1,\, \cdots,\, x_m$.

The root vector $X_\ell \,=\,\exp(x_\ell)V_\ell$ is an eigenvector for $H_{m+1}$
with eigenvalue, say $\lambda$. Now we use the formula for Lie derivative of
vector fields
$$
[H,\, \exp(\chi)V]\,=\, \exp(\chi)(H(\chi)V+[H,\, V])\, .
$$
Notice that any two vector fields which are combinations of
$\partial_{x_1},\, \cdots,\, \partial_{x_m}$ with
coefficients that are independent of $x_1,\, \cdots, \, x_m$ actually commute.
Hence
$$
[H_{m+1},\, X_\ell]\,=\, \exp(x_\ell)(H_{m+1}(x_\ell)(U_1+U_2)+
[H_{m+1},\, U_1]+ [H_{m+1},\, U_2])
$$
$$
=\, \exp(x_\ell)(H_{m+1}(x_\ell)(U_1+U_2)+ [H_{m+1},\, U_2])\,=\,
\lambda\cdot \exp(x_\ell)(U_1+U_2)\, .
$$
Thus
\begin{equation}\label{e1}
H_{m+1}(x_\ell)(U_1+U_2)+ [H_{m+1},\, U_2]\,=\, \lambda(U_1+U_2)\, .
\end{equation}

Now $H_{m+1}$ is a combination of $\partial_{x_1},\,
\cdots,\, \partial_{x_m}$ and $U_2$ is a combination of
$\partial_{x_{m+1}},\,\cdots,\, \partial_{x_N}$, and all the coefficients of both
the vector fields are independent of $x_1,\, \cdots,\, x_m$. Consequently,
$[H_{m+1},\, U_2]$ is a combination of $\partial_{x_1},\,
\cdots,\, \partial_{x_m}$. Taking the two sides of \eqref{e1}
modulo $\partial_{x_1},\, \cdots,\, \partial_{x_m}$ we see that
$$
H_{m+1}(x_\ell)\cdot U_2\,=\, \lambda\cdot U_2\, .
$$
But $H_{m+1}(x_\ell)$ is not a constant. Thus $U_2$ must be identically zero.
Now \eqref{e1} reads
$$
H_{m+1}(x_\ell)\cdot U_1\,=\, \lambda\cdot U_1\, ,
$$
so $U_1$ must also be identically zero. Consequently, $X_\ell$
must be identically zero. This is a contradiction. Therefore, we conclude that
the rank of $H_1,\, \cdots ,\, H_m,\, H_{m+1}$ is $m + 1$.

To complete the proof, take a standard set of generators
$\{X_{\alpha_i},\, X_{-\alpha_i}\}_{i=1}^n$ corresponding to the simple roots
$\alpha_1,\,\cdots,\, \alpha_n$. The Cartan subalgebra is spanned by
$$
H_{\alpha_i}\, =\, [X_{\alpha_i},\, X_{-\alpha_i}]\, ,\ \ \ i\,=\, 1,\, \cdots,\, n\, .
$$
As seen above, we can introduce coordinates $x_1,\, \cdots,\, x_N$ in which, because of
commutativity of the $H_{\alpha_1},\, \cdots,\, H_{\alpha_n}$ and the rank of
the system of vector fields $\{ H_{\alpha_1},\, \cdots,\, H_{\alpha_n}\}$ being $n$, the
vector field $H_{\alpha_i}$ becomes $\partial_{x_i}$ for all $1\,\leq\, i\,
\leq\, n$.

The root vectors corresponding to the simple roots and their negatives are
$$
X_{\alpha_i}\,=\, \exp(\chi_{_i}) V_i\, ,\ \ X_{-\alpha_i}\,=\, \exp(-\chi_{_i}) W_i\, ,
$$
where $\chi_{_i}$ are linear functions given by $\chi_{_i} \,=\,\sum_{j=1}^n c_{ij}x_j$
and $c_1,\, \cdots,\, c_n$ are entries of the Cartan matrix defined by the simple system
of roots and the vector fields $V_i,\, W_i$, $i\,=\, 1,\, \cdots,\, n$,
have coefficients that are independent of the coordinates $x_1,\,\cdots,\, x_n$.

As the Cartan matrix is nonsingular, we can make a linear change of variables
$\widetilde{x}_i \,=\,\chi_{_i}$, $i\,=\, 1,\, \cdots,\, n$, while leaving the remaining
variables, if there is any, unchanged.

Thus in these variables the root vectors corresponding to the simple roots and their
negatives have the form stated in the theorem.

\section{Applications of Theorem \ref{thm1}}

Let us apply Theorem \ref{thm1} to determine local representations of ${\rm
sl}(n,{\mathbb R})$ as vector fields on ${\mathbb R}^n$. This will be used in an
essential way in finding invariant systems of differential equations as well as all
semisimple algebras of vector fields in ${\mathbb R}^n$, with $n\, \leq\, 3$.

As is well known, any semisimple Lie algebra with a split Cartan subalgebra is
generated by copies of ${\rm sl}(2, {\mathbb R})$ -- one copy for each node of the
Dynkin diagram. The copies on adjacent nodes form a rank two subalgebra; moreover for
root systems with only single bonds, the rank two subalgebras are copies of ${\rm
sl}(3, {\mathbb R})$ or ${\rm sl}(2, {\mathbb R})\times{\rm sl}(2, {\mathbb R})$.

Thus, it is important to determine representations of subalgebras as vector fields
on ${\mathbb R}^N$ of the type specified in the following proposition.

\begin{proposition}\label{prop1}
Any representation of ${\rm sl}(3, {\mathbb R})$ as vector fields on
${\mathbb R}^N$ whose root spaces for the
simple roots are of the form $X_\alpha \,=\, \exp(x)U$, $X_\beta \,=\,
\exp(y)V$, where $U$ and $V$ are constant
vector fields must be -- up to multiplicative non-zero constants -- of the form
$$
X_\alpha \,=\, \exp(x)(\partial_x +U_1)\, ,\ \ \
X_\beta \,=\, \exp(y)(\partial_x-\partial_y +V_1)
$$
or
$$
X_\alpha \,=\, \exp(x)(\partial_x-\partial_y +U_1)\, ,
\ \ \ X_\beta \,=\, \exp(y)(\partial_y +V_1)\, ,
$$
where $U_1 ,\, V_1$ are constant vector fields whose
$\partial_x$ and $\partial_y$ components vanish
for the basis
$$\partial_x\,=\, \partial_{x_1},\,
\partial_y\,=\, \partial_{x_2},\,
\partial_{x_3},\, \cdots,\, \partial_{x_N}\, .
$$
\end{proposition}

\begin{proof}
For vector fields $U$, $V$ and functions $\chi$, $\psi$,
$$
[\exp(\chi)U,\, \exp(\psi)V]\,=\, \exp(\chi+\psi)(U(\psi)\cdot V-V(\chi)\cdot U+
[U,\, V])\, .
$$
In particular, if $U$ and $V$ are constant vector fields on ${\mathbb R}^N$, then
$$
[\exp(\chi)U,\, \exp(\psi)V]\,=\, \exp(\chi+\psi)(U(\psi)\cdot V-V(\chi)\cdot U)\, .
$$

First note that if $X \,=\,\exp(x)A,\, Y\,=\, \exp(-x)B$ generate
${\rm sl}(2, {\mathbb R})$, with X, Y eigenvectors for their commutator with opposite
eigenvalues, and $$A \,=\, a\partial_x+C\, ,\ \ \
B \,=\, c\partial_x+D$$ are constant vector fields, with
$C$, $D$ supported outside $\partial_x$, then both $a$ and $c$
must be nonzero, otherwise $X,\, Y$ would generate a solvable algebra. By
``supported outside $\partial_x$'' we mean that the expression does not contain the
term $\partial_x$; we will employ this terminology.

Using the notation of the statement of the proposition, write
$$
U\,=\, a\partial_x+b\partial_y+U_1\, ,\ \
V\,=\, c\partial_x+d\partial_y+V_1\, ,
$$
where $U_1,\, V_1$ are constant vector fields supported outside
$\partial_x,\, \partial_y$; thus they
play no role in the commutation relations of the Lie algebra generated by
$X_\alpha$ and $X_\beta$. Therefore, to prove the proposition, we may ignore them. As
remarked above, we
may assume that $a$, $c$ are both non-zero and therefore we may assume
$$
X_\alpha\,=\, \exp(x)(\partial_x+\lambda\cdot\partial_y
+U_1)\, ,\ \ X_\beta\,=\, \exp(y)(\mu\cdot\partial_x+
\partial_y+V_1)\, .
$$
The commutation relations are not affected by ignoring the
constant vector fields supported outside $\partial_x,\,
\partial_y$ and we ignore them henceforth.
So we may assume that
$$
X_\alpha\,=\, \exp(x)(\partial_x+\lambda\cdot\partial_y)
\, ,\ \ X_\beta\,=\, \exp(y)(\mu\cdot\partial_x+
\partial_y)\, .
$$

The positive roots of a system of type $A_2$ are $\alpha$, $\beta$,
$\alpha+\beta$. Hence $[X_\alpha,\,X_\beta]$ commutes with both $X_\alpha$,
$X_\beta$. Now
\begin{equation}\label{x1}
[X_\alpha,\,X_\beta]\,=\, \exp(x+y)((\lambda-1)\mu\partial_x+(1-\mu)
\lambda\partial_y)\, .
\end{equation}
Its commutator with $X_\alpha$ is:
$$
\exp(2x+y)(\lambda-1)\mu(\partial_x+\lambda\partial_y)
-(\lambda+1)((\lambda-1)\mu\partial_x
+(1-\mu)\lambda\partial_y))\, .
$$

Hence $[X_{\alpha+\beta},\, X_\alpha]\,=\, 0$ implies that
$$
(\lambda-1)\mu-(\lambda-1)(\lambda+1)\mu\,=\,0\,=\, (\lambda-1)\lambda\mu-
(\lambda+1)\lambda(1-\mu)\, .
$$
If both $\lambda$ and $\mu$ are nonzero, then these equations give
$\lambda\,=\, 1\, =\,\mu$.
But then
$$
[X_\alpha,\, X_\beta]\, =\, 0
$$
from \eqref{x1}.

Thus one of $\lambda$ and $\mu$ is zero. If $\lambda\,=\, 0$,
we may take $X_\alpha\, =\, \exp(x) \partial_x$. Consequently,
$$
X_{\alpha+\beta}\,=\, \exp(x+y)(\lambda-1)\mu\partial_x\, ,
$$
and we may therefore take $X_{\alpha+\beta}\,=\,\exp(x+y)\partial_x$.

Using
$$
[X_{\alpha},\, X_{\alpha+\beta}]\,=\, \exp(x+2y)(\mu+1)\partial_x\,=\, 0
$$
gives $\mu\,=\, -1$. This yields the first representation
$$
X_{\alpha}\,=\, \exp(x)\partial_x\, ,\ \
X_{\beta}\,=\, \exp(y)(\partial_x-\partial_y)\, .
$$

If $\mu\,=\, 0$, then we may take $X_{\beta}\,=\, \exp(y)\partial_y$.
Therefore, from \eqref{x1},
$$
[X_{\alpha},\, X_{\beta}]\,=\,\exp(x+y)\lambda\partial_y\, .
$$

Therefore, $\lambda\,\not=\, 0$, and we can take
$$
X_{\alpha+\beta}\,=\, \exp(x+y)\partial_y\, .
$$
Its commutator with $X_{\beta}$ is zero, while its commutator with $X_{\alpha}$
is $\exp(2x+y)(1+\lambda)\partial_y$. As this commutator is
zero, we must have $\lambda\,=\, -1$. This gives the second representation
$$
X_{\alpha}\,=\, \exp(x)(\partial_x-\partial_y)\, ,\ \
X_{\beta}\,=\, \exp(y)\partial_y\, .
$$
This completes the proof.
\end{proof}

\begin{corollary}\label{cor2}
Any representation of ${\rm sl}(n+1, {\mathbb R})$ as vector fields on ${\mathbb R}^n$,
$n\,\geq\,2$, is equivalent by point transformations to the
following two representation given by the following root vectors for the simple roots
and their negatives and both the representations are equivalent by a point
transformation:
\begin{enumerate}
\item $X_1\,=\, \exp(x_1)\partial_{x_1}$,
$X_{\alpha_i}\,=\, \exp(x_i)(\partial_{x_i}-\partial_{x_{i-1}})$,\,
$2\, \leq\, i\, \leq\, n$,
$$
X_{-\alpha_i}\,=\, \exp(-x_i)(
\partial_{x_i}-\partial_{x_{i+1}}),\ 1\, \leq\, i\, \leq\, n-1,\
X_{-\alpha_n}\,=\, \exp(-x_n) \partial_{x_n}\, .
$$

\item $X_{\alpha_i}\,=\, \exp(x_i)(
\partial_{x_i}-\partial_{x_{i+1}})$,\, $1\, \leq\, i\, \leq\, n-1$, $X_{\alpha_n}\,=\, \exp(x_n)
\partial_{x_n}$,
$$
X_{-\alpha_1}\,=\, \exp(-x_1)
\partial_{x_1},\ X_{-\alpha_i}\,=\, \exp(-x_i)(
\partial_{x_i}-\partial_{x_{i-1}}),\ 2\, \leq\, i\, \leq\, n\, .
$$
\end{enumerate}
\end{corollary}

\begin{proof}
Use Proposition \ref{prop1} and Theorem \ref{thm1} inductively for the chain
of subalgebras
$$
{\rm sl}(2, {\mathbb R})\, \subset\, {\rm sl}(3, {\mathbb R})\, \subset\,
\cdots \, \subset\, {\rm sl}(n, {\mathbb R})\, \subset\,{\rm sl}(n+1, {\mathbb R})
$$
and the commutator relations $[X_{-r},\, X_s]\,=\,0$ for simple and unequal roots $r$, $s$
and $[X_r,\, X_s]\,=\, 0$ if $r+s$ is not a root and $r\,\not=\, -s$.
\end{proof}

\begin{corollary}\label{cor1}\mbox{}
\begin{enumerate}
\item Real analytic semisimple algebras of vector fields in $\mathbb R$ can be only
split and of type $A_1$.

\item Real analytic semisimple Lie algebras of vector fields in ${\mathbb R}^2$
can only be real forms
of algebras of types $A_1$, $A_1\times A_1$ or $A_2$.

\item Real analytic semisimple Lie
algebras of vector fields in ${\mathbb R}^3$ -- apart from the
types listed in (2) -- can only be real forms of algebras of
types $B_2$, $A_2\times A_1$ or $A_3$.
\end{enumerate}
\end{corollary}

The proof of this corollary together with all the real forms will appear elsewhere.

\section{Applications to systems of ordinary differential equations}

As already mentioned in the introduction, to find linearizable equations, one could
look at systems of ordinary differential equations with known semisimple Lie algebras of symmetries and find
their realizations as vector fields and determine the joint invariants in a suitable
prolongation.

Here, we restrict ourselves to the algebras ${\rm sl}(3, {\mathbb R})$ and ${\rm
sl}(4, {\mathbb R})$ that are known to be the symmetry algebras of $y'' \,=\, 0$ and
$y''\, =\, 0\,=\, z''$, respectively. Therefore, we need to find all representations
of ${\rm sl}(3, {\mathbb R})$ and ${\rm sl}(4, {\mathbb R})$ as vector fields on
${\mathbb R}^3$ and ${\mathbb R}^4$. Declaring one of the variables to be the
independent variable and considering
the remaining as dependent, we need to find the invariants
in the second prolongation. If, in the canonical coordinates for an algorithmically
computable abelian subalgebra there are linear systems present, then those
coordinates would give the linearizing coordinates.

We will implement this program in this section.

\begin{proposition}\label{prop2}\mbox{}
\begin{enumerate}
\item If a second order ordinary differential equation of the form $y'' \,=\, f(x, y, y')$ has ${\rm sl}(3,
{\mathbb R})$ as its symmetry algebra, then the canonical coordinates for the maximal
abelian subalgebra of ad-nilpotent elements of maximal rank give the linearizing
coordinates for the given ordinary differential equation and in these coordinates the equation is $y'' \,=\, 0$.

\item If a system of second order ordinary differential equations of the form
$y''(x) \,=\, f(x, y, z, y', z')$, $z''(x) \,=\, g(x, y, z, y', z')$
has ${\rm sl}(4, {\mathbb R})$ as
its symmetry algebra, then the canonical coordinates for the maximal abelian
subalgebra of ad-nilpotent elements of maximal rank give the linearizing coordinates
for the given system of ordinary differential equations. In these coordinates the system becomes
$y''(x)\,=\, 0\,=\, z''(x)$.
\end{enumerate}
\end{proposition}

Before giving a proof of Proposition \ref{prop2}, we recall, for the benefit of
non-specialists, how on can compute prolongations of vector fields ab initio --
following Lie \cite[p. 261--274]{Li}. For generalities, the reader
is referred to \cite{Ib} and \cite{Ol}.

First of all, a vector field – in local coordinates $(x_1,\, \cdots,\, x_n)$ is a sum
of vector fields $f_i{\partial}_{x_i}$. Therefore one needs to only find
prolongations of such fields and add them to get the prolongation to any desired
order. Secondly, the prolongations are obtained by repeated use of the chain rule.

Bearing this in mind, the proof of Proposition \ref{prop2} is a consequence of the canonical
representations of ${\rm sl}(3,
{\mathbb R})$ and ${\rm sl}(4,{\mathbb R})$. Although the first part is already in
\cite{Li}, we have given details in this case also as the same ideas regarding
linearizing coordinates and where they live work also for systems of ordinary differential equations.

\section{Proof of Proposition \ref{prop2}}

\subsection{Proof of part (1)}

Let $L\,=\, {\rm sl}(3, {\mathbb R})$. It has only one representation -- up to point
transformations -- given by
$$
X_\alpha\,=\, \exp(x)({\partial}_{x} -{\partial}_y),\
X_\beta\,=\, \exp(y){\partial}_{y},\
X_{-\alpha}\,=\, \exp(-x){\partial}_{x},\
X_{-\beta}\,=\, \exp(-y)({\partial}_{x} -{\partial}_y)\, .
$$
A maximal ad-nilpotent subalgebra corresponds to all the root spaces for the positive
roots. Thus, besides $X_\alpha$ and $X_\beta$, it has, as a basis,
$X_{\alpha+\beta}\,=\, \exp(x+y){\partial}_{y}$. The
algebra generated by $X_\alpha$ and $X_{\alpha+\beta}$ is abelian and its rank
is two. It is a maximal
abelian algebra made up of root vectors. We write the full algebra in terms of the
canonical coordinates for this abelian algebra.

The canonical coordinates for this abelian algebra are given by
$\exp(x)({\partial}_{x} -{\partial}_y)\,=\, {\partial}_u$,
$\exp(x+y){\partial}_y\,=\, {\partial}_v$. The
differentials of $u$, $v$ are dual to the fields
${\partial}_u$, ${\partial}_v$. Thus, in these coordinates, the root
spaces work out to be
$$
X_\alpha\,=\, \partial_u,\ X_\beta\,=\, u\partial_v,\
X_{-\alpha}\,=\,u(u\partial_u+v\partial_v),\ X_{-\beta}\,=\,v\partial_u\, .
$$
For convenience, we write $u$, $v$ as $x$, $y$ respectively.

As vector fields on a plane,
$$
X_\alpha\,=\, (1,\, 0),\ X_\beta\,=\, (0,\, x),\ X_{-\alpha}\,=\,
(x^2,\, xy),\ X_{-\beta}\,=\, (y,\, 0)\, .
$$
The corresponding flows -- up to order $\epsilon$ -- are given by:
$$
(\widetilde{x},\, \widetilde{y})\,=\, (x+\epsilon,\, y),\
(\widetilde{x},\, \widetilde{y})\,=\, (x,\, y+\epsilon x),\
(\widetilde{x},\, \widetilde{y})\,=\, (x+\epsilon x^2,\, y+\epsilon xy),\
(\widetilde{x},\, \widetilde{y})\,=\, (x+\epsilon y,\, y)\, .
$$
To find the prolonged action, we need to declare one of the variables as independent
and the remaining variables as dependent and then we have to look at how the first
and second derivatives are transformed by these flows. If
$$
(\widetilde{x},\, \widetilde{y})\,=\, (x+\epsilon,\, y)\, ,
$$
then
$$
\frac{d\widetilde y}{d\widetilde x}\,=\, \frac{dy}{dx}\ \ \text{ and } \ \
\frac{d^2\widetilde y}{d{\widetilde x}^2}\,=\, \frac{d^2y}{dx^2}\, .
$$
Consequently, the second prolongation of $X_\alpha$ has no components in
the variables $y',\, y''$ and therefore it is given by $X^{(2)}_\alpha\,=\,
\partial_x$.

If $(\widetilde{x},\, \widetilde{y})\,=\, (x,\, y+\epsilon x)$,
then
$$
\frac{d\widetilde y}{d\widetilde x}\,=\, \frac{dy+\epsilon dx}{dx}
\,=\, \frac{dy}{dx}+\epsilon\, ,\ \
\frac{d^2\widetilde y}{d{\widetilde x}^2}\,=\,
\frac{d}{d\widetilde x}(\frac{d\widetilde y}{d\widetilde x})\,=\,
\frac{\frac{d}{dx}(\frac{dy}{dx}+\epsilon)}{\frac{d\widetilde x}{dx}}
\,=\, \frac{d^2y}{dx^2}\, ,
$$
therefore $X^{(2)}_\beta\,=\, x\partial_y+\partial_{y'}$.

For computing the second prolongation of $X_{-\alpha}\,=\,
x(x\partial_x+y\partial_{y})$, it is convenient
to compute the second prolongations of $x^2\partial_x$ and
$xy\partial_y$ and add them.
Thus if $(\widetilde{x},\, \widetilde{y})\,=\, (x+\epsilon x^2,\, y)$, then
$$
\frac{d\widetilde y}{d\widetilde x}\,=\, \frac{dy}{dx}-2\epsilon x\frac{dy}{dx}+
\cdots\, ,\ \frac{d^2\widetilde y}{d{\widetilde x}^2}\,=\, \frac{d^2y}{dx^2}
-2\epsilon(\frac{dy}{dx}+x\frac{d^2y}{dx^2})-2\epsilon x\frac{d^2y}{dx^2}+\cdots\, .
$$
Therefore, the second prolongation of $x^2\partial_x$ is
\begin{equation}\label{l1}
x^2\partial_x-2xy'\partial_{y'}-2(y'+2xy'')\partial_{y''}\, .
\end{equation}

If $(\widetilde{x},\, \widetilde{y})\,=\, (x,\, y+\epsilon xy)$, then
$$
\frac{d\widetilde y}{d\widetilde x}\,=\, \frac{dy}{dx}+\epsilon (y+x\frac{dy}{dx})\,
 , \ \frac{d^2\widetilde y}{d{\widetilde x}^2}\,=\, \frac{d^2y}{dx^2}
+\epsilon(2\frac{dy}{dx}+x\frac{d^2y}{dx^2})\, .
$$
Therefore the second prolongation of $xy \partial_y$ is
\begin{equation}\label{l2}
xy\partial_y +(y+xy')\partial_{y'}+(2y'+xy'')\partial_{y''}\, .
\end{equation}
Adding \eqref{l1} and \eqref{l2} gives the second prolongation
of $X_{-\alpha}$ namely,
$$
X^{(2)}_{-\alpha}\,=\, x^2\partial_x+xy\partial_y + (y-xy')\partial_{y'}-3xy''
\partial_{y''}\, .
$$

Now suppose we have a second order equation of the form
$y''\,=\, f(x,y,y')$ -- invariant under $X_\alpha$ and $X_\beta$; this means that
it is at least invariant under $X^{(2)}_{\alpha}\,=\,\partial_x$
and $X^{(2)}_{\beta}\,=\,x\partial_y+\partial_{y'}$. Then
it is invariant under $[X^{(2)}_{\alpha},\, X^{(2)}_{\beta}]\,=\,
\partial_y$. Thus it is of the form $y''\,=\,f(y')$. Hence invariance under
$X^{(2)}_{\beta}$ shows that the equation is $y''-k\,=\, 0$, where $k$ is
a constant. Applying $X^{(2)}_{-\alpha}$ to the
equation we must have $-3xy'' \,=\, 0$ on the hypersurface $y'' - k \,= \,0$.
Therefore $k\,=\, 0$ and the equation is invariant under the full algebra.

Hence the only second order equation with symmetry algebra ${\rm
sl}(3, {\mathbb R})$ is -- in suitable coordinates $y''\,=\,0$.
These coordinates are the canonical coordinates of the maximal
abelian subalgebra of maximal rank in the nil-radical of a Borel
subalgebra of the symmetry algebra ${\rm sl}(3, {\mathbb R})$. The
nil-radical of the Borel subalgebra can be algorithmically
determined from a single ad-nilpotent element as shown in
\cite{AABGM}. This completes the proof of the first part.

\subsection{Alternative proof of part (1)}

We note that while an equation $y''\,=\, f(x,y,y')$ is transformed to a similar
equation under point transformations, this is no longer true for a system of two
such equations. We give below an alternative argument for part (1) whose ideas work
also for systems.

A maximal ad-nilpotent subalgebra corresponds to all the root spaces for the
positive roots. Thus, besides $X_\alpha$ and $X_\beta$, it has, as a basis,
$X_{\alpha+\beta}\,=\, \exp(x+y)\partial_y$. The algebra generated by
$X_\alpha$ and $X_{\alpha+\beta}$ is abelian and of rank two. It is a
maximal abelian algebra $A$ made up of root vectors. We write the full algebra in
terms of the canonical coordinates $(x,\, y)$ for this abelian algebra.

In the representation of ${\rm sl}(3, {\mathbb R})$ in $V({\mathbb R}^2)$ extended to
the 2nd prolongation, we have independent variables $x,\, y,\, y',\, y''$. Then, by the
Implicit function theorem (see, e.g. \cite{dO})
the invariant nonsingular hypersurfaces $H$ have defining equations
in which -- because of invariance under $\partial_x,\, \partial_y$ -- the variables
$x,\, y$ do not appear. Thus the equations for $H$ can only be
$$
y''-f(y')\,=\, 0 \ \ \ \text{ or }\ \ \ y'-g(y'')\,=\, 0\, .
$$
The second equation cannot occur because of the vector field $X_\beta\,=\,x\partial_y$
whose second prolongation is $x\partial_y+\partial_{y'}$.
Consequently, we have only one possibility $y''-f(y')\,=\, 0$.

Applying $x\partial_y+\partial_{y'}$ we have $f(y')\,=\, k$.
Hence the equation is $y''-k\,=\, 0$.

Applying $X^{(2)}_{-\alpha}\,=\, x^2\partial_x+xy\partial_{y}+(y-xy')\partial_{y'}-
3xy''\partial_{y''}$ to the
equation gives $xy''\,=\, 0$ and $y''\,=\, k$. Hence if $k\, \not=\, 0$,
then the function $x$ would
vanish identically on $H$ and this contradicts that $H$ is a hypersurface. Thus the
only possibility is that H is defined by $y''\,=\,0$.

Finally, as $X^{(2)}_{-\beta}\,=\, y\partial_{x}-(y')^2\partial_{y'}-3y'y''
\partial_{y''}$, it follows that $H$ is also invariant
under $X^{(2)}_{-\beta}$. Therefore as $H$ is invariant under the generators
$X^{(2)}_{\alpha}$, $X^{(2)}_{-\alpha}$, $X^{(2)}_{\beta}$ and $X^{(2)}_{-\beta}$,
it is invariant under ${\rm sl}(3, {\mathbb R})$.

\subsection{Proof of part (2)}

As the form of invariant systems is not in general invariant under point
transformations, we need to approach this problem in a more geometric way, combined
with detailed information on the root vectors and their prolongations.

First of all we note that in the prolonged space with coordinates $x,\, y,\, z,\,
y',\, z',\, y'',\, z''$ if we have a system of equations
$$
F(x, y, z, y', z', y'', z'')\,=\, 0\,=\, G(x, y, z, y', z', y'', z'')
$$
and the rank is two at a point $p$ of the subset $F\,=\, 0\,=\, G$,
then locally, the set $M$ defined by these equations is a five dimensional submanifold of
the extended space with coordinates $x,\, y,\, z,\,
y',\, z',\, y'',\, z''$.

By the implicit function theorem, the condition of rank being two near a given point
$p$ of $M$ implies that we can solve for two of the variables explicitly in terms of the
remaining variables.

If $M$ is invariant under translations $\partial_x,\, \partial_y,\, \partial_z$, then
none of these variables can be $x$, $y$ or $z$.
Thus we have to choose two of the variables from $y',\, y'',\, z',\, z''$.
Moreover, the resulting equations cannot involve $x$, $y$ or $z$ as free variables,
again because of invariance under $\partial_x,\, \partial_y,\, \partial_z$.

If $M$ is also invariant under $X\,=\, x\partial_y$ then as the second prolongation of
$X$ is $x\partial_y+\partial_{y'}$, none of these equations can have $y'$ as
an independent variable or a dependent variable.

So we have the following possibilities:

\begin{enumerate}
\item[(i)] Dependent variables are $y'',\, z'$, independent variable is $z''$
-- and equations are $y''\,=\, f(z'')$, $z'\,=\, g(z'')$.

\item[(ii)] Dependent variables are $z'$, $z''$, independent variable is
$y''$ -- and equations are $z'\,=\, f(y'')$, $z''\,=\, g(y'')$.

\item[(iii)] Dependent variables are $y''$, $z''$, independent variable is
$z'$ -- and equations are $y''\,=\, f(z')$, $z''\,=\, g(z')$.
\end{enumerate}
To determine these equations in the case at hand, we need detailed information
about the root vectors and their second prolongations.

We choose coordinates adapted to the structure of the given Lie algebra. In the case
of ${\rm sl}(4, {\mathbb R})$ the coordinates representing one independent and two
dependent variables will be the canonical coordinates of a canonically defined
subalgebra, uniquely determined up to conjugation.

By Corollary \ref{cor2}, the Lie algebra $L\,=\, {\rm sl}(4, {\mathbb R})$ has only one representation
as vector fields in three dimensions; up to point transformations it is given by given
by the root vectors
$$
X_\alpha\,=\, \exp(x)\partial_x,\, X_\beta\,=\, \exp(y)(\partial_y-\partial_x),\,
X_\gamma\,=\, \exp(z)(\partial_z-\partial_y)
$$
$$
X_{-\alpha}\,=\, \exp(-x)(\partial_x-\partial_y),\,
X_{-\beta}\,=\, \exp(-y)(\partial_y-\partial_z),\,X_{-\gamma}\,=\, \exp(-z)\partial_z\, .
$$
A maximal abelian subalgebra of ad-nilpotent elements of geometric rank 3 is
$$
X_{\alpha+\beta+\gamma}\,=\, \exp(x+y+z)\partial_x,\, X_{\beta+\gamma}\,=\,
\exp(y+z)(\partial_y-\partial_x),\,
X_\gamma\,=\, \exp(z)(\partial_z-\partial_y)\, .
$$

The canonical coordinates for this algebra are given by solving
$$
\exp(x+y+z)\partial_x\,=\, \partial_u,\, \exp(y+z)(\partial_y-\partial_x)\,
=\, \partial_v,\, \exp(z)(\partial_z-\partial_y)\,=\, \partial_w\, .
$$
Solving this system, we have
$$
u\,=\, -\exp(-x-y-z),\, v\,=\, -\exp(-y-z),\, w\,=\, -\exp(-z)\, .
$$
Thus $\partial_x\,=\, -u\partial_u$, $\partial_y-\partial_x\,=\, -v\partial_v$
and $\partial_z-\partial_y\,=\, -w\partial_w$. Also
$$
\exp(z)\,=\, -\frac{1}{w},\, \exp(y)\,=\, \frac{w}{v},\,
\exp(x)\,=\, \frac{v}{u}\, .
$$
Therefore, in these coordinates, ignoring signs,
$$
X_\alpha\, =\, v\partial_u,\, X_\beta\, =\, w\partial_v,\,X_\gamma\, =\, \partial_w,\,
X_{-\alpha}\, =\, u\partial_v
$$
$$
X_{-\beta}\, =\, v\partial_w,\, X_{-\gamma}\, =\, w(u\partial_u+v\partial_v+
w\partial_w)\, .
$$
We need to compute the second prolongation of these fields. To do this, we have to
declare one of these variables as the independent and the remaining variables as the
dependent variables.

We take $u$ as the independent variable and $v$, $w$ as the dependent variables.

Following conventions, we re-label $u\,=\, x$, $v\,=\, y$ and $w\,=\, z$.
Thus, our fields are
$$
X_\alpha\, =\, y\partial_x,\, X_\beta\, =\, z\partial_y,\,X_\gamma\, =\, \partial_z,\,
X_{-\alpha}\, =\, x\partial_y
$$
$$
X_{-\beta}\, =\, y\partial_z,\, X_{-\gamma}\, =\, z(x\partial_x+y\partial_y+
z\partial_z)\, .
$$
We find the second prolongations by using the chain rule as in \cite[p.~261--274]{Li}.
We obtain
\begin{itemize}
\item $X^{(2)}_\alpha\,=\, y\partial_x-(y')^2\partial_{y'}-y'z'\partial_{z'}-
3y'y''\partial_{y''}-(y''z'+2y'z'')\partial_{z''}$,

\item $X^{(2)}_\beta\,=\, z\partial_y+z'\partial_{y'}+z''\partial_{y''}$,

\item $X^{(2)}_{-\alpha}\,=\, x\partial_y+\partial_{y'}$,

\item $X^{(2)}_{-\beta}\,=\, y\partial_z+y'\partial_{z'}+y''\partial_{z''}$,

\item $X^{(2)}_{\gamma}\,=\,\partial_z$.
\end{itemize}
To find $X^{(2)}_{-\gamma}$, it is convenient to find the second prolongations of
$zx\partial_x$, $zy\partial_{y}$, $z^2\partial_{z}$ and add them. We have:
\begin{itemize}
\item $(zx\partial_x)^{(2)}\,=\, zx\partial_x-y'(xz'+z)\partial_{y'}-
z'(xz'+z)\partial_{z'}-(2y''(xz'+z)+y'(xz''+2z'))\partial_{y''}-(3xz'z''+
2(z')^2+2zz'')\partial_{z''}$,

\item $(zy\partial_y)^{(2)}\,=\, yz\partial_y+(y'z+yz')\partial_{y'}+(y''z+
2y'z'+yz'')\partial_{y''}$,

\item $(z^2\partial_z)^{(2)}\,=\, z^2\partial_z+2zz'\partial_{z'}+2((z')^2+zz'')
\partial_{z''}$.
\end{itemize}

We can now determine all the invariant systems. We have the following possibilities:

{\bf Case (1):}\, Dependent variables are $y'',\, z'$, independent variable is
$z''$, and equations are $y''\,=\, f(z'')$, $z'\,=\, g(z'')$.

Applying $X^{(2)}_{-\beta}\,=\, y\partial_z+y'\partial_{z'}+y''\partial_{z''}$, we must have
$y'\,=\,y''g'(z'')$. Thus on $M$ we have one more functionally independent equation
$y'\,=\, f(z'')g'(z'')$ and dimension of $M$ would decrease. Hence this case does not
arise.

{\bf Case (2):}\, The dependent variables are $z',\, z''$, the
independent variable is $y''$, and the equations are $z'\,=\,
f(y'')$, $z''\,=\, g(y'')$. Applying
$X^{(2)}_{-\beta}\,=\, y\partial_z+y'\partial_{z'}+y''\partial_{z''}$
to $z'\,=\, f(y'')$ gives $y'\,=\, 0$
along the solution space $M$ and its dimension would decrease. Thus, this case
also does not arise.

{\bf Case (3):}\, The dependent variables are $y''$, $z''$, the
independent variable is $z'$, and the equations are $y''\,=\,
f(z')$, $z''\,=\, g(z')$. Applying $X^{(2)}_{-\beta}\,=\,
y\partial_z+y'\partial_{z'}+y''\partial_{z''}$ we have
$y'f'(z')\,=\,0$, $y''\,=\, y'g'(z')$.

If $f'(z')$ is not identically zero along the solution space then $y'\,=\,0$ and the
dimension of $M$ would decrease. Hence $f'(z')\,=\, 0$ along $M$.

Thus $y''\,=\, k$ along $M$ and $k\,=\, y' g'(z')$. If $k\,=\, 0$,
then $y'$ cannot vanish along $M$ and therefore $y' g'(z')\,=\,0$
implies $g(z')\,=\, l$ and $M$ is defined by $y''\,=\, 0$,
$z''\,=\, l$. If $k\, \not=\, 0$, then we would have an extra
equation
$$y'\,=\,\frac{k}{g'(z')}$$ and the dimension of M would go down.

Consequently, the only possibility for $M$ is that it is defined
by $y''\,=\, 0$, $z''\,=\, l$.

Applying $X^{(2)}_\alpha$ to $z''\,=\, l$ we get $-2y'z''\,=\, 0$
on $M$ and if $y'$ is identically
zero on some open set of $M$, then again we would get an independent equation
$y'\,=\,0$ and the local equations for $M$ would be $y''\,=\, 0$, $z''\,=\, l$,
$y'\,=\, 0$ and the dimension
of $M$ would go down. Thus, $z''$ is identically $0$ on $M$ and the equations for
$M$ are indeed $y''\,=\, 0\,=\, z''$.

It remains to check that this is indeed an invariant submanifold of
${\rm sl}(4, {\mathbb R})$ in these
coordinates. From the equations for the second prolongations of the generators, it
only remains to check invariance under $X^{(2)}_{-\gamma}$.

We have
$$X^{(2)}_\alpha\,=\, y\partial_x-(y')^2\partial_{y'}-y'z'\partial_{z'}-
3y'y''\partial_{y''}-(y''z'+2y'z'')\partial_{z''}\, ,\ \ X^{(2)}_{\gamma}\,=\,
\partial_z$$
$$X^{(2)}_{-\alpha}\,=\, x\partial_y+\partial_{y'}\, ,\ \
X^{(2)}_{-\beta}\,=\, y\partial_z+y'\partial_{z'}+
y''\partial_{z''}
$$
and $X^{(2)}_{-\gamma}$ is a sum of fields and we need to just consider the contributions
in the $\partial_{y''}$ and $\partial_{z''}$ directions; these contributions are:
$$
-(2y''(xz'+z)+y'(xz''+2z'))\partial_{y''}-(3xz'z''+2(z')^2+2zz'')\partial_{z''}\, ,
$$
$$
(y''z+2y'z'+yz'')\partial_{y''}
\ \ \text{ and }\ \ 2((z')^2+zz'')\partial_{z''}
$$
from the formulas given above; their sum vanishes on the set
$y''\,=\, 0\,=\, z''$. This completes the proof of the proposition.

\section*{Acknowledgements}

I.B. is supported by a J. C. Bose Fellowship. F.M. is indebted to the N.R.F. of 
South Africa for research grant support.


\begin{thebibliography}{ZZZZZZZ}

\bibitem[AABGM]{AABGM} S. Ali, H. Azad, I. Biswas, R. Ghanam and M. T. Mustafa,
Embedding algorithms and applications to differential equations, arXiv:1603.08167.

\bibitem[AKM]{AKM} M. Ayub, M. Khan and F. M. Mahomed, Algebraic
linearization criteria for systems of ordinary differential equations, {\it Nonlinear
Dynam.} {\bf 67} (2012), 2053--2062.

\bibitem[ABGM]{ABGM} H. Azad, I. Biswas, R. Ghanam and M. T. Mustafa, On computing
joint invariants of vector fields, {\it Jour. Geom. Phys.} {\bf 97} (2015), 69--76.

\bibitem[Bo]{Bo} A. Borel, {\it Lie groups and linear algebraic groups},\\
http://hkumath.hku.hk/~imr/records0001/borel.pdf.

\bibitem[dO]{dO} O. de Oliveira, The implicit and the inverse function
theorems: easy proofs, {\it Real Analysis Exchange} {\bf 39} (2014), 207--218.

\bibitem[GL]{GL} A. Gonz\'alez-L\'opez, Symmetries of linear systems of second-order
ordinary differential equations, {\it Jour. Math. Phys.} {\bf 29} (1988), 1097--1105.

\bibitem[HN]{HN} J. Hilgert and K.-H. Neeb, {\it Structure and geometry of Lie
groups}, Springer Monographs in Mathematics, Springer, New York, 2012.

\bibitem[Ib]{Ib} N.H. Ibragimov, {\it Elementary Lie group analysis and ordinary
differential equations}, Wiley series in Mathematical Methods in Practice, 4, John
Wiley~\&~Sons, Ltd. Chichester, 1999.

\bibitem[Kn]{Kn} A. W. Knapp, {\it Lie groups beyond an introduction}, Second edition,
Progress in Mathematics, 140. Birkh\"auser Boston, Inc., Boston, MA, 2002.

\bibitem[Ko]{Ko} B. Komrakov, Transformation groups and geometric structures (on
some Sophus Lie results today),
http://repository.kulib.kyoto-u.ac.jp/dspace/bitstream/2433/64064/1/1150-06.pdf

\bibitem[Li1]{Li} S. Lie, Vorlesungen \"uber Differentialgleichungen mit bekannten
infinitesimalen Transformationen, BG Teubner, 1891.

\bibitem[Li2]{Li2} S. Lie, Theorie der Transformationsgruppen, Vol. 3,
https://eudml.org/doc/202686.

\bibitem[Ol]{Ol} P. J. Olver, {\it Applications of Lie groups to differential
equations}, Second edition, Graduate Texts in Mathematics, 107, Springer-Verlag, New
York, 1993.

\end{thebibliography}
\end{document}